\documentclass[leqno,11pt]{amsart}
\usepackage{amssymb,latexsym,amsmath,amscd,epsfig,amsthm, amsfonts}

\theoremstyle{theorem}
\newtheorem{lemma}{Lemma}[section]
\newtheorem{theorem}{Theorem}[section]
\newtheorem*{corollary}{Corollary}
\newtheorem*{facta}{Theorem A}
\newtheorem*{fact}{Theorem B}

\newtheorem{proposition}{Proposition}[section]

\numberwithin{equation}{section}
\theoremstyle{remark}
\newtheorem*{example}{Example}

\newtheorem*{remark}{Remark}

\newcommand{\C}{{\mathbf C}}
\newcommand{\T}{{\mathbf T}}
\newcommand{\D}{{\mathbf D}}

\newcommand{\R}{{\mathbf R}}
\newcommand{\vf}{{\varphi}}

\newcommand{\Harm}{{\rm Harm}}
\newcommand{\const}{{\rm const}}
\newcommand{\supp}{{\rm supp}}

\newcommand{\E}{{\rm E}}
\newcommand{\cE}{{\mathcal E}}
\newcommand{\HH}{{\mathcal K}}

\newcommand{\Rr}{{\mathcal R}}
\newcommand{\cO}{{\mathcal O}}

\newcommand{\F}{{\mathcal F}}

\newcommand{\Df}{{\mathcal H}}
\newcommand{\beq}{\begin{equation}}
\newcommand{\eeq}{\end{equation}}
\newcommand{\beqs}{\begin{equation*}}
\newcommand{\eeqs}{\end{equation*}}
\newcommand{\beg}{\begin{gather}}
\newcommand{\eeg}{\end{gather}}

\newcommand{\Dvaan}{{2^{2^n}}}
\newcommand{\Dvaak}{{2^{2^k}}}

\newcommand{\Dvak}{2^k}
\newcommand{\Dvakk}{2^{k+1}}

\newcommand{\ve}{{\varepsilon}}

\begin{document}

\title
[Radial oscillation ]
{Radial oscillation of harmonic functions in the Korenblum class }
\author{Yurii Lyubarskii}
\address{Department of Mathematical Sciences, Norwegian University of Science and Technology, NO-7491, Trondheim, Norway}
\email{yura@math.ntnu.no} 
 \author {Eugenia Malinnikova}
 \address{Department of Mathematical Sciences, Norwegian University of Science and Technology, NO-7491, Trondheim, Norway}
\email{eugenia@math.ntnu.no}

\begin{abstract}
We study radial behavior of harmonic functions in the unit disk belonging to the Korenblum class. We prove that functions
which admit two-sided Korenblum estimate   either oscillate or have slow growth along almost all radii.
\end{abstract}

\thanks{ The authors  were partly supported by the Research
Council of Norway grant 185359/V30 and    by  ESF grant HCAA}


\maketitle

\section{Introduction  and main results}

We study   radial behavior of functions harmonic 
in the unit disk $\D$.  Let
$\Harm (\D)$ stay  for the set of all real-valued functions  harmonic 
in   $\D$ and $|E|$  for the  normalized linear Lebesgue measure of a set $E$
 on the unit circle $\T$.  It follows from the classical result of Lusin and Privalov
 see e.g. \cite{P}, that there exist  functions $u\in  \Harm(\D)$ which tend to infinity
 along almost all radii, while non-tangential growth may occur 
 only on subsets of  the  unit circle having zero linear measure.

 An important generalization of this result is due to Kahane and Katznelson 
 \cite{KK}. They proved that for any function $v(r)$, $0<r<1$  such that 
 $v(r)\nearrow \infty$ as $r\nearrow 1$ there exists a function 
$u\in  \Harm(\D)$ such that $u(re^{i\theta})\to \infty$  as  $r\nearrow 1$ and also 
\beq
\label{eq:majorant}
|u(z)|<v(|z|),  \ z\in \D. 
\eeq

It is known (see  \cite{BLMT,EM}),  that for a wide class of 
 majorants $v$ realtion  \eqref{eq:majorant} yields 
 \[
\liminf_{r\nearrow 1}  \frac 
{u(re^{i\phi})}{v(r)}   \leq  0, \  \mbox{and}  \   
       \limsup_{r\nearrow 1}  \frac 
{u(re^{i\phi})}{v(r)}   \geq  0
\]
for almost all values of $\phi\in (-\pi,\pi)$. In other words, for almost all $\phi$ the values 
$u(re^{i\phi})$  "oscillate"  between $\pm v(r)$. In this article we study this oscillation
in more details.

We restrict ourselves for definiteness  to functions which belong  to the classical (harmonic)
 Korenblum class $\HH$ consisting of all real functions $u\in \Harm(\D)$
 satisfying
 \beq
\label{eq:K}
u(z)\le \log \frac{e}{1-|z|}, \ z\in \D.
\eeq
  This class  was introduced and studied in \cite{K}. We refer the reader to 
 \cite{HKZ, S, MT, BL} for further properties of functions  in $\HH$, including their behavior near the boundary.

 One of the typical examples here is the  function
\beq
\label{eq:super}
u_A(z)=\Re \sum_n A^n z^{2^{A^n}},
\eeq
where $A\ge 2$ is an integer.
It is easy to see that $|u_A(z)|\ge c|\log(1-|z|)|$ on a large portion of 
the unit disk, and along almost each radius this function oscillates between $c|\log(1-|z|)|$ and $-c|\log(1-|z|)$. 
We give a quantitative description of this oscillation and prove that 
such oscillation occurs  for \textsl{every} harmonic function $u$ such that $u, -u\in \HH$.

 More precisely for a function $u\in \HH$,  we   consider the weighted average
\beq
\label{eq:01}
I_u(R,\vf) = \int_{1/2}^R \frac{u(re^{i\vf})}
               {(1-r)\left (\log \frac{1}{1-r}\right )^2}dr, \quad R\in (0,1), \  \vf\in(-\pi,\pi).
\eeq
 A straightforward calculation shows that \eqref{eq:K} gives
\[
I_u(R,\vf) \lesssim C\log \log \frac{1}{1-R}.
\]
Here and in the sequel we use notation $ a\lesssim b$ for the statement: 
 there exists an absolute constant $c$ such that $a\le cb$; 
 we also write $a\simeq b$ if $a\lesssim b$ and $b\lesssim a$.

 If both $u\in \HH$  and $-u\in \HH$, the behavior of $I_u(R,\vf)$
 is controlled by the law  of the iterated logarithm:
 
  \begin{theorem}
\label{th:main}
There exists $K$ such that
if $u$ is a harmonic function in $\D$  satisfying
\beq
\label{eq:below}
|u(z)|\le \log\frac{e}{1-|z|},
\eeq
then
\beq
\label{eq:main}
\limsup_{R\nearrow 1}I_u(R,\phi)\left(\log\log\frac1{1-R}\log_4\frac{1}{1-R}\right)^{-1/2}\le K
\eeq
for almost every $\phi\in(-\pi,\pi]$, (here $\log_4 x=\log\log\log\log x$) and also
\beq
\label{eq:mean}
\int_{-\pi}^\pi I_u(R,\phi)^2d\phi\le K\log\log\frac1{1-R}.
\eeq
\end{theorem}

This result is sharp in the following sense: there exists $\alpha>0$ such that 
the function $u=u_2$  defined by
\eqref{eq:super} with $A=2$  satisfies
 \beq
\label{eq:mainex}
\limsup_{R\nearrow 1}I_u(R,\phi)\left(\log\log\frac1{1-R}\log_4\frac{1}{(1-R)}\right)^{-1/2} >\alpha
\eeq
for almost every $\phi\in(-\pi,\pi]$.
 
Theorem \ref{th:main} allows  one to obtain  more  precise estimates of the radial growth  
for the case when both $u$ and $-u$ belong to $\HH$.
\begin{proposition}
\label{pr:1}
Let $u\in\HH$ and $-u\in\HH$ then for any $a<1/2$
\beq
\liminf_{r\rightarrow 1}\frac{u(re^{i\phi})(\log|\log(1-r)|)^{a}}{|\log(1-r)|}\le 0,	
\eeq
for almost every $\phi\in(-\pi,\pi]$.
\end{proposition}

The proof of Theorem \ref{th:main} is based on   approximation of $I_u(R,\phi)$ 
by the sum of discrete martingales and then application of the law of the iterated 
logarithm for martingales. 
Our approach is based on  the ideas developed by D.L.~Burkholder and R.F.~Gundy \cite{BG},
 C.Y.~Chang, J.M.~Wilson, and T.H.~Wolf \cite{CWW}, N.~Makarov \cite{M}, 
 J.M.~Anderson, L.D.~Pitt \cite{AP}, and R.~Ba\~{n}uelos, I.~Klemes, and C.N.~Moore \cite{BKM}, 
 see also the references therein. 
  We also refer the reader to the monographs \cite{B, BM}.
  However approximation of  harmonic functions by martingales, which became classical by now,
   does not yield the desired approximation  of the weighted average $I_u(R,\vf)$.
  In particular the law of the iterated logarithm for trigonometric series does not say much 
  about   oscillations of our model function (\ref{eq:super}).
  Instead we consider the martingale approximation of the function  $I_u(R,\phi)$  itself.
  We use 
 "super-dyadic" martingales corresponding to the algebras generated by intervals of length $2^{-2^{n}}$ 
 which  seem to be more appropriate for our purposes.
 One of  possible ways to  think about this construction is to consider a suitable
 martingale transform of the "classical" martingale of the harmonic function $u$ 
 thinned out to the "super-dyadic"  algebras.

These results can be directly applied to lacunary  series.   
 We use a result of J.-P.~Kahane, M.~Weiss, and G.~Weiss, \cite{KWW}
 in order to describe all functions  
    $u$ which are  represented by the series
\beq
\label{eq:gap}
u(z)=\Re\sum_{k=1}^\infty a_{n_k}z^{n_k}, \quad n_{k+1}/n_k>\lambda>1 \ 
\mbox{for each $k$},
\eeq
%where $n_{k+1}/n_k>\lambda>1$ for each $k$, 
and belong to $\HH$  in terms of the coefficient sequence  $\{a_{n_k}\}$. 
It follows from this description that if a function $u$ of the form 
\eqref{eq:gap} belongs to $\HH$ then 
  $-u\in\HH$.

\begin{corollary}
Let $u\in\HH$   be represented by a lacunary series (\ref{eq:gap}). Then (\ref{eq:main}) holds almost everywhere,
here $K$ depends only on  $\lambda$ from the gap condition and on $C$ in (\ref{eq:K}).
\end{corollary}

The situation changes drastically when we consider  functions $u\in \HH$ that satisfy only one-sided estimate \eqref{eq:K} instead of the two-sided estimate
\eqref{eq:below}. This follows from the statement below

\begin{theorem}
\label{th:contr}
The series
 \beq
\label{eq:01cnt}
u(z)= \Re \sum_{n>1} 2^n \frac{z^{2^{2^n}}}{z^{2^{2^n}} -1}
\eeq
converges uniformly on compact sets in $\D$ to a function in $\HH$,
and, for almost every $\phi\in (-\pi,\pi]$,
\beq
\label{eq:04}
\liminf_{R\nearrow 1} I_u(R,\phi)\left(\log\log\frac1{1-R}\right)^{-1}>0
\eeq
\end{theorem}

\medskip

The article is organized as follows. In the next Section we collect some results 
on premeasures and (discrete) martingales that are used later. Sections 3 and 4 
deal with the martingale approximation of $I_u(R,\phi)$ and contain the proof 
of Theorem \ref{th:main}. An example showing that Theorem \ref{th:main} is sharp 
is given in Section 5, there we also describe harmonic functions in the Korenblum 
class given by lacunary series. Theorem \ref{th:contr} is proved in the last section.  

{\bf Acknowledgment} The main  part  of this work was done when 
the authors have been visited the Mathematical Department of
University of California, Berkeley. It is our pleasure to thank 
the Department and Prof. Donald Sarason  for kind hospitality. 
 
%%%%%%%%%%%%%%%%%%%%%%%%%%%%%%%%%%%%%%%%%%%%%%%%%%
%%%%%%%%%%%%%%%%%%%%%%%%%%%%%%%%%%%%%%%%%%%%%%%%%% 

\section{Preliminaries}

%%%%%%%%%%%%%%%%%%%%%%%%%%%%%%%%
\subsection{Representation of functions from $\HH$}

We refer the reader to the classical article
\cite{K} and to the monograph \cite{HKZ}. Let   $\Rr$ be the set of all 
(open, closed, half-closed)
arcs on the unit circle $\T$.  A function
$\mu: \Rr \to \R$  is a {\em premeasure} if it satisfies
\begin{itemize}
\item $\mu(I_1\cup I_2)=\mu(I_1)+\mu(I_2)$ when $I_1\cap I_2=\emptyset$ and $I_1\cup I_2\in\Rr$,
\item   $\mu(\T)=0$
\item $\lim_{n\rightarrow\infty}\mu(I_n)=0$ whenever $I_1\supset I_2\supset...$ and $\cap_nI_n=\emptyset$.
\end{itemize}
We say that $\mu$ is $\kappa$-{\em bounded from above}, if in addition
\beq
\label{eq:muabove}
\mu(I)\lesssim |I|\log\frac e{|I|}
\eeq
for every $I\in\Rr$.

Everywhere below we assume (for simplicity) that  $u(0)=0$.
Let also
\[
P(re^{i\phi})= \frac{1-r^2}{|e^{i\phi}-r|^2}
\]
be the standard Poisson kernel.
It is proved in \cite{K}  that each $u\in \HH$  can be represented
as the Poisson integral:
\beq
\label{eq:05}
u(re^{i\phi})= \int_{-\pi}^\pi P(re^{i(\phi-\theta)}) d\mu(\theta)
\eeq
with respect to some $\kappa$-bounded from above  premeasure $\mu$.
The integral in the right-hand side should be understood as
\beq
\label{eq:06}
\int_{-\pi}^\pi P(re^{i(\phi-\theta)}) d\mu(\theta)=
     \int_{-\pi}^\pi \left (P(re^{i(\phi-\theta)})\right )'_\theta \mu((e^{i\theta},0) ) d\theta,
\eeq
here  $(e^{i\theta},0)$  stays for the arc of $\T$  which connects $0$
and $e^{i\theta}$.

In the case when $u$ satisfies the two-sided estimate \eqref{eq:below} the corresponding premeaasure is
$\kappa$-bounded from above and below:
\beq
\label{eq:mubelow}
|\mu(I)|\lesssim |I|\log\frac e{|I|}.
\eeq

\subsection{Martingales}

In this part we recall the basic notions and facts about  (super-dyadic) martingales
(on the unit circle) which will be used in the sequel. We  follow  mainly
\cite{St}, see also \cite{Stbook}.

Given $n>0$ let $\cE_n$  be the set of all  super-dyadic intervals (on $\T$) of length $2^{-2^n}$
and
$\F_n $ be the  $\sigma$-algebra generated by $\cE_n$.
Respectively we denote $\cE=\cup \cE_n$.
A function $f:\T\to \C$  is measurable with respect to $\F_n$  if it is constant on each $I\in \cE_n$.
Given any   $f:\T\to \C$ we define its expectation with respect to $\E(f|\F_n)$ as
the measurable (with respect to $\F_n$)  function 
\beq
\label{eq:07}
\E(f|\F_n)= \sum_{I\in \cE_n}
\left (\frac 1 {|I|}\int_I f(t)dt  \right ) \textbf{1}_I ,
\eeq
here $\textbf{1}_I$ denotes the characteristic function of $I$.
A sequence of function $\{f_n\}$, $f_n: \T\to \C$ is called a {\em martingale} if
\begin{itemize}
\item
$f_n$ is $\F_n$ measurable
\item
$
E(f_n| \F_{n-1})= f_{n-1}
$
\end{itemize}

The {\em martingale differences} and the {\em square function} 
of a martingale $\{f_n\}$ are defined by 
\beq
\label{eq:08}
d_j=f_j-f_{j-1} \ \mbox{and} \ s_n=\left (\sum_{j=1}^n E(d_j^2|\F_{j-1}) \right )^{1/2},
\eeq
respectively, see e.g. \cite{St}.
Remark, that our martingales are super-dyadic so the formula  for $s_n$   differs from 
one usually used for dyadic martingales,
we refer the reader to \cite{BM} for related discussion.
We also denote
\beq
\label{eq:09}
u_n=\left ( 2 \log\log s_n^2 \right )^{1/2}.
\eeq

The following statement is a special case of Theorems 1  and 2 in \cite{St}.
\begin{fact}
 Let $\cO=\{\phi\in(-\pi,\pi]): s_n(\phi)\to \infty \ \mbox{as} \ n\to \infty\}$    and
$|d_n| \lesssim 1$. Then
\beq
\label{eq:10}
\limsup_{n\to \infty} \frac {f_n(\phi)}{s_n(\phi)u_n(\phi)} \le 1,
\eeq
 for almost all $\phi \in \cO$.
\end{fact}

We always identify $\phi$ and the point $e^{i\phi}\in\T$.

\begin{example} Let a premeasure $\mu$ satisfy \eqref{eq:mubelow}.
Denote
\beq
\label{eq:11}
g_n=\sum_{I\in\F_n}\textbf{1}_I\frac {\mu(I)} {|I|},
\eeq
\beq
\label{eq:12}
d_j=2^{-j}(g_j-g_{j-1}) \  \mbox{and} \   f_n = \sum_{j=1}^n d_j.
\eeq
 Then the martingale $\{f_n\}$  meets the conditions of Theorem A.
\end{example}

\begin{remark} The function $u$ in Theorem \ref{th:main}  admits representation
\eqref{eq:05}, so for each arc $I=(e^{i(\theta-\delta)}, e^{i(\theta+\delta)})\subset \T$ 
 the average
$|I|^{-1}\mu(I)$  can be considered as an approximation of   $u((1-\delta)e^{i\theta})$.
Taking this into account one  can observe that the functions $f_n$  from the above example
can be viewed as integral sums for $I_u(1-2^{-2^n}, \cdot)$. 
This martingale can be used as a hint in order to guess how Theorem \ref{th:main}
should be formulated.
However this  prove Theorem \ref{th:main} we need a more developed 
martingale construction.
 \end{remark}
%%%%%%%%%%%%%%%%%%%%%%%%%%%%%%%%%%%%%%%%%%%%%%%%%%%%%%%%%%%%%
%%%%%%%%%%%%%%%%%%%%%%%%%%%%%%%%%%%%%%%%%%%%%%%%%%%%%%%%%%%%%%%%

\section{Atomic decomposition}

%%%%%%%%%%%%%%%%%%%%%%%%%%%%%%%%%%%%%%%%%%

Let a function $u\in \Harm(\D)$ satisfy \eqref{eq:below} and $u(0)=0$. In this section we decompose the function $I_u(R,\phi)$  into sum of atoms.
Such decomposition (following for example the scheme from \cite{BM})  will lead us to
a martingale approximation of $I_u(R,\phi)$.

\subsection{Preliminary decomposition}
Denote
\beq
\label{eq:13}
r_j=1-2^{-2^j},\ j=0,1,... \ .
\eeq

Let $r_n\le R< r_{n+1}$. We than have
\begin{multline}
\label{eq:14}
I_u(R,\phi)=\int_{1/2}^R\frac{u(re^{i\phi})dr}{(1-r)(\log(1-r))^2}= \\
\sum_{j=1}^n \underbrace{\int_{r_{j-1}}^{r_j}\frac{u(re^{i\phi})}{(1-r)(\log(1-r))^2}dr }_{v_j(\phi)}+\int_{r_n}^R\frac{u(re^{i\phi})dr}{(1-r)(\log(1-r))^2}= \\
=\sum_{j=1}^{n-1}\underbrace{2v_{j+1}(\phi)-v_{j}(\phi)}_{w_j(\phi)}
+ \underbrace{\left(
2v_1(\phi)- v_n(\phi)+ \int_{r_{n}}^R  \frac{u(re^{i\phi})dr}{(1-r)|\log(1-r)|^2} \right)}_{q_n(\phi)}.
\end{multline}
It follows from \eqref{eq:below} that $|q_n(\phi)|\lesssim 1$. Therefore in order to prove Theorem
\ref{th:main}  it suffices to consider
\beq
%\label{eq:15}
J_n(\phi)= \sum_1^{n-1}w_j(\phi).
\eeq
We will approximate the sequence $\{J_n\}$   by the sum of super-dyadic martingales.

Using representation \eqref{eq:05} we obtain 
\beq
\label{eq:15}
v_j(\phi)=\int_{-\pi}^\pi A_j(\phi-\theta)d\mu(\theta), \
w_j(\phi)=\int_{-\pi}^\pi B_j(\phi-\theta)d\mu(\theta), \
\eeq
where
\beq
\label{eq:16}
A_j(\psi)=\int_{r_{j-1}}^{r_j}\frac{P(re^{i\psi})dr}{(1-r)(\log(1-r))^2},\quad B_j(\psi)=2A_j(\psi)-A_{j-1}(\psi).
\eeq
Clearly
\beq
\label{eq:17}
\int_{-\pi}^\pi A_j(\phi)d\phi=2^{-j} \ \mbox{and} \ \int_{-\pi}^\pi B_j(\phi)d\phi=0,
\eeq
so the kernels $B_j$ possess the cancellation property.

\subsection{Construction of atoms}
The functions $B_j(\phi)$ are concentrated mainly in the intervals around zero whose length 
is of order $2^{-2^{j-4}}$. We approximate them by functions $\tilde{B}_j(\phi)$ 
which are supported in the corresponding intervals. The following (technical) lemma estimates  
 the error of such approximation.
\begin{lemma}
\label{l:appr}
For every $j$ there exists an even function $\tilde{B_j}(\psi)$ such that
\[\supp(\tilde{B_j})\subset(-\frac 1 {64} 2^{-2^{j-4}},\frac 1 {64} 2^{-2^{j-4}})=:J_j\quad {\rm and}\quad
B_j=\tilde{B_j}\  {\rm on}\  \frac12 J_j,\]
\[
\int_{-\pi}^\pi\tilde{B_j}=\int_{\T}B_j=0,
\quad
|(B_j-\tilde{B_j})'|\lesssim 2^{-2j}.
\]
Moreover, the following estimates hold
\begin{equation}
\label{eq:cor1}
(a)\ |\tilde{B}'|\lesssim 2^{-2j}2^{2\cdot2^j},\quad
(b)\ |\tilde{B}''|\lesssim 2^{-2j}2^{3\cdot2^j},\quad
(c)\ |\tilde{B}_j(\alpha)\alpha\log\frac{1}{\alpha}|\lesssim 2^{-j};
\end{equation}
\begin{equation}
\label{eq:cor2}
\int_{-\pi}^{\pi}|\tilde{B}'_j(\alpha)\alpha\log\frac{1}{\alpha}|d\alpha\lesssim 1.
\end{equation}
\end{lemma}

\begin{proof}  Let $r_j$  be given by \eqref{eq:13}. 
Elementary calculations show that
\[
A_j(\phi)\lesssim \begin{cases}
2^{-2j}2^{2^j},\quad &\phi\le 1-r_j,\\
2^{-2j}\phi^{-1},\quad &1-r_j<\phi\le 1-r_{j-1},\\
2^{-2j}\phi^{-2}2^{-2^{j-1}},\quad &1-r_{j-1}<\phi.
\end{cases}
\]
Further, we obtain
\[
|A'_j(\phi)|\lesssim \begin{cases}
2^{-2j}\phi 2^{3\cdot2^j},\quad &\phi\le 1-r_j,\\
2^{-2j}\phi^{-2},\quad &1-r_j<\phi\le 1-r_{j-1},\\
2^{-2j}\phi^{-3}2^{-2^{j-1}},\quad &1-r_{j-1}<\phi;
\end{cases}
\]
and
\[
|A''_j(\phi)|\lesssim \begin{cases}
2^{-2j}2^{3\cdot 2^j},\quad &\phi\le 1-r_j,\\
2^{-2j}\phi^{-3},\quad &1-r_j<\phi\le 1-r_{j-1},\\
2^{-2j}\phi^{-4}2^{-2^{j-1}},\quad &1-r_{j-1}<\phi.
\end{cases}
\]

Let $\alpha$ and $\beta$ be even smooth functions, $0\le\alpha\le1$, such that
\[\alpha=1,\quad{\rm{on}}\ (-2^{-7}2^{-2^{j-4}},2^{-7}2^{-2^{j-4}});\]
\[\supp(\alpha)\subset[-2^{-6}2^{-2^{j-4}},2^{-6}2^{-2^{j-4}}],\quad |\alpha'|\lesssim 2^{2^{j-4}},\  {\rm{and}}\  |\alpha''|\lesssim 2^{2^{j-3}};
\]
\[
\int_{-\pi}^{\pi}\beta=1,\quad \supp(\beta)\subset[-2^{-6}2^{-2^{j-4}},-2^{-7}2^{-2^{j-4}}]\cup[2^{-7}2^{-2^{j-4}},2^{-7}2^{-2^{j-4}}];
\]
\[
|\beta|\lesssim 2^{2^{j-4}},\quad |\beta'|\lesssim 2^{2^{j-3}},\quad {\rm{and}}\ |\beta''|\lesssim 2^{2^{j-2}}.\]
Then we define
\[\tilde{B}_j=\alpha B_j+\beta\int_{-\pi}^\pi (B_j-\alpha A_j).\]
Clearly $\tilde{B}_j$ is an even function.  The required estimates now follows by a direct inspection.
\end{proof}

%\texttt{I am leaving the lemma as it is for the moment, but perhaps it will be rewritten later}

Now let
\beq
\label{eq:17a}
\tilde{w}_j(\phi)=\int_{-\pi}^{\pi}\tilde{B_j}(\phi-\theta)d\mu(\theta).
\eeq

\begin{lemma}
\label{l:app}
The following inequality holds
\[
|w_j(\phi)-\tilde{w}_j(\phi)|\lesssim 2^{-2j}\]
for any $j$ and $\phi\in(-\pi,\pi)$.
\end{lemma}

\begin{proof}
Clearly (\ref{eq:below}) implies that $|\mu(I)|\lesssim 1 $ for any interval $I$ on the circle.
Since $B_j-\tilde{B}_j$ is an even function and $\mu(\T)=0$ we have
\[
|w_j(\phi)-\tilde{w}_j(\phi)|=\left|\int_{-\pi}^{\pi}(B_j-\tilde{B}_j)(\phi-\theta)d\mu(\theta)\right|
\]\[=\left|\int_0^{\pi}(B_j-\tilde{B}_j)'(\theta)\mu(\phi-\theta,\phi+\theta)\right|\lesssim 2^{-2j}.
\]
\end{proof}
This lemma  implies
\beq
\label{eq:appr2}
\left|\sum_1^n w_j(\phi)-\sum_1^n\tilde{w}_j(\phi)\right|\lesssim 1.
\eeq

We divide the circle into $64\cdot2^{2^k}$ non-overlapping arcs of length
$\frac{1}{64}2^{-2^k}$ and define by $\Df_k$  the collection of these arcs.
Let also $\Df=\cup_k\Df_k$.
We obtain
\beq
\label{eq:lambda}
\tilde{w}_j(\phi)=\int_{-\pi}^\pi\tilde{B}_j(\phi-\theta)d\mu(\theta)=\sum_{I\in\Df_{j-4}}\int_I\tilde{B}_j(\phi-\theta)d\mu(\theta)
=\sum_{I\in\Df_{j-4}}\lambda_I(\phi).
\eeq
Functions $\lambda_I$ are our atoms. For  each $I\in \Df_k$, $k>6$ we have
\[
\supp \lambda_I\subset 3I;
\]
and
\beq
\label{eq:19}
\int_{3I}\lambda_I=\int_{-\pi}^\pi\lambda_I(\phi)d\phi=\int_{-\pi}^\pi\int_I\tilde{B}_j(\phi-\theta)d\mu(\theta)d\phi=0.
\eeq

%%%%%%%%%%%%%%%%%%%%%%%%%%%%%%%%%%%%%%%%%%%%%%%%%%%%%%
%%%%%%%%%%%%%%%%%%%%%%%%%%%%%%%%%%%%%%%%%%%%%%%%%%%%%

\section{From atoms to martingales}

\subsection{Construction of martingales}
Relation \eqref{eq:appr2} gives a decomposition of  $\sum_{j=1}^n\tilde{w}_j$ into the sum of atoms $\lambda_I$,
  $I\in \Df$. It may happen that 
  $\supp \lambda_{I_1} \cap \supp \lambda_{I_2} \neq \emptyset$ 
  for (neighboring) $I_1,I_2\in \Df$, $|I_1|=|I_2|$.

Given  $\omega \in \T$  and an arc $I\in \T$ we  denote
$\omega I= \{\omega \zeta; \zeta \in I\}$. Respectively
$\omega \cE= \{\omega I; I\in \cE \}$, $\omega \cE_n= \{\omega I; I\in \cE_n \}$.
The  lemma below follows from a more general statement
 Lemma 2.1.2 in \cite{BM}, see  also
\cite{CWW} (we adjust the formulation for our setting.)

\begin{lemma}
\label{le:app}
There exists a finite partition $\Df=\cup_{s=1}^N V^{(s)}$
and a set of points
$\{\omega_s\}_{s=1}^N\subset \T$ such that
 $V^{(s)}\cap V^{(t)}=\emptyset, s\neq t$ and, 
for each $I\in V^{(s)}\cap \Df_k $, there exists $I'\in \omega_s \cE_k $
for which $\supp\lambda_I\subset 3I\subset I'$. In addition if, for some $s$ and $k$,
$I_1,I_2\in V^{(s)}\cap \Df_k$, $I_1\neq I_2$  then $I_1'\cap I_2'=\emptyset$.
\end{lemma}

Now for each  $s=1,2,\ldots\, , N$  and $n\ge 1$ we define
\beq
\label{eq:15m}
\Lambda_n^{(s)}(\phi)=\sum_{I\in V^{(s)}, 64|I|\ge 2^{-2^n}}\lambda_I(\phi).
\eeq
Then
\beq
\label{eq:16m}
\sum_{j=1}^n\tilde{w_j}(\phi)=\sum_{s=1}^N\Lambda_{n-4}^{(s)}(\phi).
\eeq
We consider the corresponding  martingales with the sequence of (shifted) super-dyadic $\sigma$-algebras %as for the case of a premeasure in 1.1),
\beq
\label{eq:17m}
f_n^{(s)}=E(\Lambda_n^{s}|\omega_s\F_{n} ).
\eeq

\subsection{Estimate of the martingale approximation}
Our first aim is to estimate the error
\[|f_n^{(s)}-\Lambda_n^{(s)}|\le \sum_{I\in V^s, 64|I|\ge 2^{-2^n}}|E(\lambda_I|\omega_s\F_{n})-\lambda_I|.\]

\begin{lemma}
\label{l:A}
The following inequality  holds
\beq
\label{eq:18}
|f_n^{(s)}(\phi)-\Lambda_n^{(s)}(\phi)|\lesssim 1 .
\eeq
\end{lemma}

\begin{proof}
First we prove that 
\beq
\label{eq:19b}
|\lambda_I(\phi)|\lesssim 1\quad {\rm{ and}}\quad
|\lambda'_I(\phi)|\le |I|^{-64}
\eeq
for any $I\in\Df$ and any $\phi\in(-\pi,\pi]$;

Indeed let $\omega_\phi(\theta)=\int_{\phi}^\theta d\mu=\mu(\phi,\theta)$ 
and let $I=(\alpha, \beta)\in\Df_{j-4}$. We have
\[
|\lambda_I(\phi)|=\left|\int_I\tilde{B}_j(\phi-\theta)d\mu(\theta)\right|\le\]
\[
\left|\int_I\tilde{B}'_j(\phi-\theta)\omega_\phi(\theta)d\theta\right|+
|\tilde{B}_j(\phi-\alpha)\omega_\phi(\alpha)|+|\tilde{B}_j(\phi-\beta)\omega_\phi(\beta)|.\]
Using inequalities (\ref{eq:cor2}) and (\ref{eq:cor1} (c)), we obtain the first estimate
in \eqref{eq:19b}
To prove the second estimate we write
\[
|\lambda'_I(\phi)|\le
\left|\int_I\tilde{B}''_j(\phi-\theta)\omega_\phi(\theta)d\theta\right|+
|\tilde{B}'_j(\phi-\alpha)\omega_\phi(\alpha)|+|\tilde{B}'_j(\phi-\beta)\omega_\phi(\beta)|\]
and use the inequalities (\ref{eq:cor1} (a)) and (\ref{eq:cor1} (b)).
%$|\tilde{B}''_j|\lesssim 2^{-2j}2^{3\cdot2^j}$ and
%$|\tilde{B}'_j|\lesssim 2^{-2j}2^{2\cdot2^j}$.

Now we have
\[|E(\lambda_I|\omega_s\F_{n} )-\lambda_I|\le 2^{-2^{n}}\max|\lambda'_I|.\]
%Relations \eqref{eq:15}, \eqref{eq:17}, and \eqref{eq:19}  imply
This inequality together with \eqref{eq:19b} imply
\begin{multline}
\label{eq:18a}
|f_n^{(s)}(\phi)-\Lambda_n^{(s)}(\phi)|\le
\sum_{k=1}^n\ \sum_{ I\in V^s\cap \Df_k, I'\ni \phi}|E(\lambda_I|\omega_s\F_{n} )(\phi)-\lambda_I(\phi)|
\le \\
\sum_{k=1}^{n-6}\ \sum_{ I\in V^s\cap \Df_k, I'\ni \phi}|E(\lambda_I|\omega_s\F_{n} )(\phi)-\lambda_I(\phi)|+ \\
\sum_{k=n-5}^n\ \sum_{ I\in V^s\cap \Df_k, I'\ni \phi}(|E(\lambda_I|\omega_s\F_{n} )(\phi)|+|\lambda_I(\phi)|)
\lesssim 
2^{-2^n}\sum_{k=1}^{n-6} 2^{2^{k+6}}+1\lesssim 1,
\end{multline}
and \eqref{eq:18} now follows.
\end{proof}

%%%%%%%%%%%%%%%%%%%%%%%%%%%%%%%%%%%%%%%%%%%%%%
%\subsection{Square function of the martingale}

For each martingale $\{f_n^{(s)}\}$ we can now estimate its square function, 
\[
s_n^{(s)}=\left(\sum_{k=1}^n E(|f_k^{(s)}-f_{k+1}^{(s)}|^2|\F_{k})\right)^{1/2}. \]

\begin{lemma}
For each $s=1,...,N$ and $n\ge 1$
\beq
\label{eq:20}
|f_n^{(s)}-f^{(s)}_{n+1}|\lesssim 1 \quad{\rm{and}}\quad s_n^{(s)}\lesssim \sqrt{n}.
\eeq

\end{lemma}

\begin{proof}
The first inequality follows from Lemma \ref{l:A} and  
\eqref{eq:19b}. The second inequality   is now straightforward. 
 
\end{proof}
%%%%%%%%%%%%%%%%%%%%%%%%%%%%%%%%%%%%%

\subsection{Mean estimates}

We first prove   inequality  \eqref{eq:mean}  from Theorem \ref{th:main}, namely
\[
%\label{eq:mean}
\int_{-\pi}^\pi I_u(R,\phi)^2d\phi\lesssim\log\log\frac1{1-R}.
\]
Let, as before, $r_n$'s be given by \eqref{eq:13}  and $r_n\le R < r_{n+1}$. 
  Then (see \eqref{eq:14}, \eqref{eq:appr2}, \eqref{eq:16}, and \eqref{eq:18})
\[
\left | I_u(R,\phi)- \sum_{s=1}^N f_{n-4}^{(s)}(\phi)  \right | \lesssim 1.
\]
Therefore it suffices to prove that
\beq
\label{eq:21}
  \int_{-\pi}^\pi f_{n-4}^{(s)}(\phi) ^2d\phi\lesssim  n  \simeq
\log\log\frac1{1-R}
\eeq
for each $s=1,2,\ldots \, ,N$.

We use the fact that the martingale differences $f_j^{(s)}(\phi)-f_{j-1}^{(s)}(\phi)$
are pairwise orthogonal.
Therefore
\begin{multline}
\label{eq:22}
\int_\T|f_{n-4}^{(s)}(\phi)|^2d\phi=
         \int_\T|\sum^{n-4}_{j=1} (         f_{j}^{(s)}(\phi)- f_{j-1}^{(s)}(\phi))|^2d\phi = \\
   \int_\T\sum_{j=1}^{n-4} |f_{j}^{(s)}(\phi)-f_{j-1}^{(s)}(\phi)|^2d\phi =
            \int_\T  \left(s_{n-4}^{(s)}(\phi)\right)^2 d\phi \lesssim n,
\end{multline}
the last inequality follows from \eqref{eq:20}. This yields  \eqref{eq:mean}.
%%%%%%%%%%%%%%%%%%%%%%%%%%%%%%%%%%%%%%%
%%%%%%%%%%%%%%%%%%%%%%%%%%%%%%%%%%%%%%%%

\subsection{The law of the iterated logarithm}
The relation \eqref{eq:main}  from Theorem \ref{th:main}  now   follows easily
from Theorem A and the estimate \eqref{eq:20}  of the square functon. Indeed, since
$|I_u(R,\phi)-\sum_{s=1}^N f_{n-4}^{(s)}(\phi)|\lesssim 1 $  it suffices to prove that for each $s$
\[
\limsup_n \frac{f_n^{(s)}}{n\log\log n} \lesssim 1
\]
almost everywhere. In case $s^{(s)}_n(\phi) \to \infty$   this follows from relations
\eqref{eq:10}  and  \eqref{eq:20}. Otherwise $s^{(s)}_n(\phi)$ stays bounded and
\[
|f^{(s)}_n(\phi)|\le \sum_{j=1}^n |f^{(s)}_{j}(\phi)-f^{(s)}_{j-1}(\phi)| \le n^{1/2}
s^{(s)}_n(\phi)^{1/2}.
\]
 
This completes the proof of Theorem \ref{th:main}. Proposition \ref{pr:1} 
formulated in the introduction follows readily.
%%%%%%%%%%%%%%%%%%%%%%%%%%%%%%%%%%%%%%%%%%%
%%%%%%%%%%%%%%%%%%%%%%%%%%%%%%%%%%%%%%%%%%%%%%
%%%%%%%%%%%%%%%%%%%%%%%%%%%%%%%%%%%%%%%%%%%%%%

 \section{Lacunary series}

\subsection{Example} We begin with an example showing that Theorem \ref{th:main} is sharp.
Let \[
u(z)=\Re\sum_{n=1}^\infty 2^nz^{2^{2^n}},\quad {\rm{and}}\quad
I_u(R,\phi)=\int_{1/2}^R\frac{u(re^{i\phi})}{(1-r)\left(\log(1-r)\right)^2}dr.\]
It is proved in  \cite{BLMT} that $u\in\HH$. We will show that 
for some $a>0$
\[
\limsup_{R\rightarrow 1}I_u(R,\phi)\left(\log\log\frac1{1-R}\log_4\frac{1}{(1-R)}\right)^{-1/2}\ge a ,\]
for almost all $\phi\in(-\pi,\pi]$. Let
\[
c_n=2^n\int_{1/2}^1\frac{r^{2^{2^n}}}{(1-r)\left(\log(1-r)\right)^2}dr,\]
and $r_k=1-2^{-2^k}$. 
Suppose that $R\in(r_k,r_{k+1})$. We prove first that  
\beq
\label{eq:app}
|I_u(R,\phi)-\sum_{j=1}^k c_j\cos(2^{2^j}\phi)|\lesssim 1.
\eeq
Indeed,
\beq
\label{eq:estI}
|I_u(R,\phi)-I_u(r_k,\phi)|\lesssim\int_{r_k}^{r_{k+1}}\frac{dr}{(1-r)|\log(1-r)|}\lesssim 1.
\eeq
Further,
\[
I_u(r_k,\phi)=\sum_{j=1}^\infty\int_{1/2}^{r_k}\frac{2^jr^{2^{2^j}}}{(1-r)(\log(1-r))^2}dr\cos(2^{2^j}\phi)=\sum_{j=1}^\infty c_{j,k}\cos(2^{2^j}\phi).\]
For $j> k$ we have
\[
c_{j,k}=\int_{1/2}^{r_k}\frac{2^jr^{2^{2^j}}}{(1-r)(\log(1-r))^2}dr\lesssim
2^{j}\left(1-2^{-2^k}\right)^{2^{2^j}}\lesssim 
2^{j}\exp(-2^{2^j-2^k})
\]
and $\sum_{j> k}c_{j,k}\lesssim 1$.
Finally, for $j\le k$ we get
\[
\left|c_{j,k}-c_j\right|=2^j\int_{r_k}^1\frac{r^{2^{2^j}}}{(1-r)(\log(1-r))^2}dr\le
2^j\int_{r_k}^1\frac{dr}{(1-r)(\log(1-r))^2}\lesssim 2^{j-k}.\]
Therefore $\sum_{j\le k}|c_{j,k}-c_j|\lesssim 1$. Inequality (\ref{eq:app}) is proved.

Now we apply the law of the iterated logarithm proved in \cite{W} to the lacunary series 
\beq
\label{eq:series}
\sum_j c_j\cos(2^{2^j}\phi).
\eeq 
First note that 
\begin{multline*}
c_j=2^j\int_{1/2}^1\frac{r^{2^{2^j}}}{(1-r)\left(\log(1-r)\right)^2}dr\leq
2^j\int_{1/2}^{r_{j-1}}\frac{r^{2^{2^j}}}{(1-r)\left(\log(1-r)\right)^2}dr+\\
2^j\int_{r_{j-1}}^1\frac{dr}{(1-r)\left(\log(1-r)\right)^2}\lesssim 2^j\left(1-2^{-2^{j-1}}\right)^{2^{{2^j}}}+2^j2^{-j+1}\lesssim 1.
\end{multline*}
and
\[
c_j\ge 2^jr_j^{2^{2^j}}\int_{r_j}^1\frac{dr}{(1-r)\left(\log(1-r)\right)^2}\gtrsim 1.\]
Thus the coefficients in (\ref{eq:series}) are bounded from $0$ and $\infty$ and 
\[
B_k^2=\frac12\sum_{j=1}^k c_j^2\simeq k.
\]
Finally
\[
\limsup_{R\rightarrow 1}\frac{I_u(R,\phi)}{\left(\log\log\frac1{1-R}\log_4\frac{1}{(1-R)}\right)^{1/2}}\gtrsim
\limsup_{k\rightarrow\infty} \frac{\sum_{j=1}^k c_j\cos(2^{2^j}\phi)}{\left(k \log\log k\right)^{1/2}}.
\]
The last upper limit is larger than a constant for almost all $\phi\in(-\pi,\pi]$ by the law of the iterated logarithm for trigonometric lacunary series.

%Suppose that $u$ satisfies (\ref{eq:K}) and
%\[
%u(z)=\Re\sum_{n=1}^\infty a_n z^n.
%\]
%It is proved in \cite{K} that $|a_{n}|\le C_1\log n$.
%The last inequality does not imply however that $u\in\HH$ even when the series
%above is lacunary as the following example shows:
%\[
%v(z)=\Re \sum_{k=1}^\infty k z^{2^k},\quad{\rm{and}}\quad v(r)\sim \left(\log\frac %e{1-r}\right)^2.\]
%%%%%%%%%%%%%%%%%%%%%%%%%%%%%%%%%%%%%%%%%%%%%%%%%%%%%%%%%%%%%%%%%%%%%%%
\subsection{Description of Korenblum harmonic functions represented by lacunary series}
Using results from \cite{K} and \cite{KWW}, we get the following description of lacunary series that represent Korenblum functions.
\begin{proposition}
Let $\{n_k\}_{k=1}^\infty$ be a sequence of positive integers such that $n_{k+1}\ge \lambda n_k$ for each $k$, where $\lambda>1$.
Let
\beq
\label{eq:lacseries}
u(z)=\Re\sum_k c_{n_k}z^{n_k},\quad c_{n_k}\in\C,\eeq
where the series converges in the unit disc.
Then the following conditions are equivalent:
\begin{itemize}
\item[(1)] there exists $\gamma_1$ such that $u(z)\le \gamma_1\log\frac{e}{1-|z|}$ for any $z\in\D$;
\item[(2)] there exists $\gamma_2$ such that $|u(z)|\le \gamma_2\log\frac{e}{1-|z|}$ for any $z\in\D$;
\item[(3)] there exists $\gamma_3$ such that $\sum_{n_k\le N}|c_{n_k}|\le \gamma_3\log N$ for any $N\ge 2$.
\end{itemize}
\end{proposition}

\begin{proof}
Let $u_r(e^{i\phi})=u(re^{i\phi})$. We show first that $(1)$ implies $(3)$.
It follows from \cite[p.209]{K} that $|c_{n_k}|\le C_1\log n_k$, where $C_1=C_1(\gamma_1)$.

Let $r_N$ be such that $r_N^N=1/2$, we have
\[
u(r_Ne^{i\phi})=\Re \sum_{n_k\le N} c_{n_k}r_N^{n_k} e^{in_k\phi}+\Re\sum_{n_k> N} c_{n_k}r_N^{n_k} e^{in_k\phi}=s_N(\phi)+t_N(\phi).\]
The reminder term can be estimated as follows
\[
\left | t_N(\phi) \right |
\le C_1\sum_{n_k>N}\log n_k 2^{-n_k/N}\le \beta(\gamma_1,\lambda)\log N.
\]
By Theorem I in \cite{KWW} there exists $\alpha=\alpha(\lambda)$ and $\phi\in(-\pi,\pi)$ such that
\[
s_N(\phi)\ge \alpha\sum_{n_k\le N}|c_{n_k}r_N^{n_k}|.\]
(This statement is elementary when $\lambda>2$; the result in \cite{KWW} is more general.)
Thus we have
\[
\sum_{n_k\le N}|a_{n_k}|\le 2\sum_{n_k\le N}|c_{n_k}r_N^{n_k}|\le 2 \alpha^{-1}(u(re^{i\phi})+|t_N(\phi)|)\le \gamma_3\log N.\]

Now assume that $(3)$ holds and let $r_N<r\le r_{N+1}$. Then $(2)$ follows readily
\[
|u(re^{i\phi})|\le\sum_{n_k\le N}|c_{n_k}|+\sum_{n_k>N}\gamma_3\log n_k r^{n_k}\le
 \gamma_2\log \frac{e}{1-r}.
\]
\end{proof}

Thus the Korenblum functions represented by (\ref{eq:lacseries}) satisfy the assumption of Theorem \ref{th:main}.

%%%%%%%%%%%%%%%%%%%%%%%%%%%%%%%%%%%%%%%%%%%%%%%%%%%%%%%%%%%%%%%%%%%%%%%%%%%%%%%%%%%%%
%%%%%%%%%%%%%%%%%%%%%%%%%%%%%%%%%%%%%%%%%%%%%%%%%%%%%%%%%%%%%%%%
%%%%%%%%%%%%%%%%%%%%%%%%%%%%%%%%%%%%%%%%%%%%%%%%%%%%%%%%%%%%%%%%

\section{Non-oscillation: Example}
\subsection{Construction} In this section we prove Theorem \ref{th:contr}.
%We define a function in $\HH$ that grows like $\log\frac1{1-|z|}$ along a large subset of $\D$.
\begin{lemma}
The series
 \beq
\label{eq:01n}
u(z)= \Re \sum_{n\ge 1} 2^n \frac{z^{2^{2^n}}}{z^{2^{2^n}} -1}=
                    \Re \sum_{n \ge 1} 2^n a_n(z)
\eeq
converges uniformly on compact sets in $\D$ to a function in $\HH$.
\end{lemma}

\begin{proof}
We split the unit disk into disjoint annuli
\beqs
% \label{eq:02}
\D=\{z:|z|<1/2\}\cup\left (\cup_n A_n\right ), \
               A_n = \{ z: 2^{-2^n} \ge 1-|z| > 2^{-2^{n+1}} \}.
\eeqs
Let $z\in A_n$ and $k>n+1$. Then (by a straightforward calculation)
\beq
\label{eq:03}
 \left | \frac{\Dvakk a_{k+1}(z)}       {\Dvak a_k(z)}
 \right |
        \lesssim 2^{ - \Dvaak}.
\eeq
  Choosing  $n$ sufficiently large we will see that
the ratio of two consequent terms with numbers larger than $n$ does not exceed some
$q<1$. This yields the uniform convergence of     \eqref{eq:01n} on compact sets in $\D$.

A direct estimate shows that $|a_{n+1}(z)|\le C $ when $z\in A_n$ and thus
\[
\sum_{k>n} 2^{k}|a_{k}(z)|  \leq C_q \log \frac 1 {1-|z|}, \ z \in A_n.
\]
To complete the proof  it suffices to estimate  $\Re \sum_{k\leq n} \Dvak a_k(z) $.
Since \[\Re \left(w/(w-1)\right)\leq 1/2\quad {\rm{for}}\ w\in \D\] we obtain:
\[
\Re \sum_{k\leq n} \Dvak a_k(z) \leq  \sum_{k\leq n} \Dvak \leq
         C \log \frac 1 {1-|z|},\ z\in A_n.
\]
  \end{proof}
\subsection{Remark} Since $u(0)=0$  we have $\int_{-\pi}^\pi u(re^{i\theta})d\theta =0$
for all $r\in (0,1)$. The summands in \eqref{eq:01n} have the form
\beq
\label{eq:03aa}
a_n(z)=\Phi(z^{2^{2^n}}); \ \Phi(w)=\frac w {w-1}, \ w\in \D.
\eeq
The function $\zeta=\Phi(w)$ maps the unit disk $\D$ onto the half-plane
$\{\zeta; \Re\zeta <1/2\}$. The asymmetry in the distribution of $\Re \Phi(w)$
forces $u$  to attain huge negative values on "small" parts of $\D$
while being positive (or slightly negative) on the remaining part of $\D$
thus maintaining zero average along the circles $|z|=r$. In this construction
the oscillation along almost all radii disappears.

  \subsection{Main lemma}
Let for brevity $N_n=\Dvaan$ so that $N_{n+1}=N_n^2$. Fix some $a>2$ and let
     \begin{gather*}
%      \label{eq:05}
        E_n =\{\phi\in (-\pi,\pi]: \cos N_n\phi > 1- n^{-a} \},
      \\
        E_{n,m}=\{\phi\in(-\pi,\pi]: \cos N_n\phi>1-N_m^{-1}\},\\
          \ F_m= \left( \cup_{n\ge m} E_n \right )
          \cup
               \left (\cup_{n< m}E_{n,m}\right ) ,\quad{\rm and}\quad  F= \cap_m F_m.
    \end{gather*}
 Since  $|E_n| \simeq n^{-a/2}$ and $\sum_{1\leq n< m}|E_{n,m}|\rightarrow 0$ ($m\rightarrow\infty$), we have $|F|=0$.
 Theorem \ref{th:contr}   now follows from  the lemma below:
  \begin{lemma}
  For each $\phi \not\in F$
  \beq
  \label{eq:04n}
  \liminf_{R \nearrow 1} \frac {I_u(R,\phi)}{\log \log \frac {1}{1-R}} > 0.
  \eeq
       \end{lemma}

 \begin{proof}
       We fix $\phi  \not \in F$. There exists $m$ such that
 $\phi  \not \in F_m$.

% Let $\alpha>0$ we will chose $\alpha$ small enough later.
For each $l$ we denote
% we define $r_k<r^*_k<r_{k+1}$ by
 %\[
 %\log r_k=-N_k^{\alpha-1},\ \log r_k^*=-(2k^aN_k)^{-1}.\]
  \beq
  \label{eq:06n}
  I_l(\phi)= \int_{1-N_l^{-1}}^{1-N_l^{-2}} \frac {u(re^{i\phi})}{(1-r) \left ( \log \frac 1 {1-r} \right )^2}dr.
  \eeq
Assume  $R=1-2^{-2^t}$, $t>m$, and  let  $ \lfloor  t \rfloor $ stay for the integer part of
$t$.
We then have
\beqs
%\label{eq:07}
I_u(R,\phi)=\sum_{l<\lfloor t \rfloor} I_l(\phi) + \int_{1-N_{\lfloor t \rfloor}^{-1}}^R\frac {u(re^{i\phi})}{(1-r) \left ( \log \frac 1 {1-r} \right )^2}dr.
\eeqs

%%%%%%%%%%%%%%%%%%%%%%%%%%%%%%%%%%%%%%
%%%%%%% 
%%%%%%%%%%%%%%%%%%%%%%%
We will prove that there exist  constants $\gamma, \Gamma>0$
and a number $l_0>m$  such that, for $l>l_0$
\beq
\label{eq:08a}
I_l(\phi)> \gamma,
\eeq
and
\beq
\label{eq:09a}
 \int_{1-N_{\lfloor t \rfloor }^{-1}}^R\frac {u(re^{i\phi})}{(1-r) \left ( \log \frac 1 {1-r} \right )^2}dr> 
 -\Gamma, \ \mbox{when} \ t>l_0.
 \eeq
Then \eqref{eq:04n} will  follow readily.
\end{proof}
\subsection{Technical details}
We have
\[
% \label{eq:10}
 I_l(\phi)=
  \sum_{j=1}^\infty \Re \int_{1-N_l^{-1}}^{1-N_l^{-2}} 2^j
        \frac {(re^{i\phi})^{N_j}} {(re^{i\phi})^{N_j}-1} \ \frac {dr}{(1-r) \left (
                \log \frac 1 {1-r}                                                           \right )^2}=
                 \sum_{j=1}^\infty i_{l,j}(\phi),
\]
  and  similarly
   \[
\int_{1-N_{\lfloor t \rfloor }^{-1}}^R\frac {u(re^{i\phi})}{(1-r) \left ( \log \frac 1 {1-r} \right )^2}dr
=\sum_{j=1}^\infty \tilde{i}_{j}(\phi).\]

Inequalities  \eqref{eq:08a} and \eqref{eq:09a}  follow from the propositions below.

\begin{proposition}
\label{prop:a}
 Let $\phi\not\in F_m$. There exist $\gamma>0$ and $l_0>m$ such that for all  $l>l_0$
\beq
\label{eq:11n}
i_{l,j}(\phi)>0,\  \mbox{\rm if} \ j<l,
\eeq
\beq
\label{eq:12n}
i_{l,l}(\phi)  > 2 \gamma  ,
\eeq
and
\beq
\label{eq:13n}
\sum_{j>l}| i_{l,j}(\phi)| < \gamma.
\eeq
\end{proposition}

\begin{proposition}  Let $\phi\not\in F_m$.  There exist $\Gamma>0$ and $l_0>m$
such that for $\lfloor t\rfloor>l_0$
\beq
\label{eq:11a}
 \tilde{i}_{j}(\phi)>0 \ \mbox{\rm if} \ j<\lfloor t\rfloor,
\eeq
\beq
\label{eq:12a}
 \tilde{i}_{\lfloor t\rfloor}>-\Gamma/2,
\eeq
and
\beq
\label{eq:13a}
 \sum_{j>\lfloor t\rfloor}|\tilde{i}_{j}(\phi)|<\Gamma/2,\ \mbox{\rm if}\ \ l>l_0.
\eeq
\end{proposition}

\noindent{\em Proof of \eqref{eq:11n} and \eqref{eq:11a}}.
Denote $c_j(\phi)= \cos N_j\phi.$
We have
\beqs
% \label{eq:14}
i_{l,j}(\phi) =
 \int_{1-N_l^{-1}}^{1-N_l^{-2}} 2^j
        \frac{r^{N_j}\left ( r^{N_j}- c_j(\phi) \right )}
       {r^{2N_j}-2r^{N_j}c_j(\phi)+1}
       \ \frac {dr}{(1-r) \left (
                \log \frac 1 {1-r}                                                           \right )^2}.
\eeqs
Then, for $r> 1-N_l^{-1}$, and $m<j<l$ we obtain
\beqs
% \label{eq:15}
r^{N_j} \geq \left ( 1- N_j^{-2}\right )^{N_j }> 1- 2^{-2^j}  > 1-j^{-a} > c_j(\phi), \ 
{\rm when} \ \phi \not\in F_m.
\eeqs
For $j\le m<l$ we have
\beqs
% \label{eq:15}
r^{N_j} \geq \left ( 1- N_m^{-2}\right )^{N_j } >1- N_m^{-1} > c_j(\phi) 
\eeqs
for any $\phi\not\in F_m$.
Hence for all $j<l$ we have $\Re a_j(re^{i\phi})>0$ when $r>1-N_l^{-1}$,  
and $\phi\not \in F_m$, which implies both \eqref{eq:11n} for all $l>m$ and  \eqref{eq:11a}.

%This in turn yields $  i_{l,j}(\phi) >0$ for $j<l$ and similarly $\tilde{i}_{j}(\phi)>0$ for %$j<\lfloor t\rfloor$.

\medskip

%\bigskip

\noindent{\em Proof of (\ref{eq:12n})}.
We fix $l>m$ and choose   $r^*\in (1-N_l^{-1}, 1-N_l^{-2})$   so that
\beq
\label{eq:12b}
(r^*)^{N_l}=1-\frac 12l^{-a}.
\eeq
For $\phi\not\in F_m$ we have $c_l(\phi)<c_l:=1-l^{-a}$ and
\begin{multline*}
% \label{eq:16}
i_{l,l}(\phi) \geq
 \int_{1-N_l^{-1}}^{1-N_l^{-2}} 2^l
        \frac{r^{N_l}\left ( r^{N_l}-c_l \right )}
       {r^{2N_l}-2r^{N_l}c_l+1}
       \ \frac {dr}{(1-r)    \left (
                \log \frac 1 {1-r}                                                           \right )^2      }  = \\
     \left (   \int_{1-N_l^{-1}}^{r^*}  +  \int_{r^*}^{1-N_l^{-2}} \right )
                2^l
        \frac{r^{N_l}\left ( r^{N_l}-c_l \right )}
            {(r^{N_l}-c_l)^2+1-c_l^2}
               \ \frac {dr}{(1-r)    \left (
                \log \frac 1 {1-r}                                                           \right )^2      } \\
                 = j_1+j_2,
\end{multline*}

Applying the inequality $x(x^2+y^2)^{-1}\ge -(2y)^{-1}$, we have
\[
% \label{eq:17a}
  \frac{ r^{N_l}- c_l }
            {(r^{N_l}-c_l)^2+ (1-c_l^2)} \ge -  l^{a/2}.
 \]
             We  also have $\log 1/(1-r) \simeq 2^l$, for $ r\in      (1-N_l^{-1}, 1-N_l^{-2})$.
%, finally (by \eqref{eq:12b})
% \beqs
%\label{eq:17}
%1-r^*\simeq 2^{-2^l}l^{-a}.
%\eeqs
Therefore
%\beqs
%\label{eq:18}
%i_1 > -2^{-l} l^{a/2} \int_{1-N_l^{-1}}^{r^*} \frac {dr}{1-r} \lesssim -2^{-l}l^{a/2} \log l.
%\eeqs
\beqs
%\label{eq:18}
j_1 \gtrsim -2^{-l} l^{a/2} \int_{1-N_l^{-1}}^{r^*} \frac {dr}{1-r} \gtrsim -2^{-l}l^{a/2} a \log l.
\eeqs
The right-hand side here can be made  arbitrary small by choosing sufficiently large $l_0$.

To complete the proof of \eqref{eq:12n}  it remains to show that
\beq
\label{eq:18an}
j_2 \simeq 1.
\eeq 
%It follows from \eqref{eq:12b}  that for $r>r^*$  and $\phi \not \in F_m$
%we  have
%\[
% 1-r < \frac 12 c_l < \frac 12 c_l(\phi).
%\]
%Therefore the point $(re^{i\phi})^{N_l}$ does not belong to the disk
%$\{w\in \D; \Re \Phi(w)<1/4\}$, here $\Phi$ is the conformal mapping defined in
%\eqref{eq:03aa}.  Hence
%For $r>r^*$ and $c_l(\phi)<c_l$ we have
%\[
% \frac{r^{N_l}\left ( r^{N_l}- c_l(\phi) \right )}
%       {r^{2N_l}-2r^{N_l}c_l(\phi)+1}  > \frac{1-\frac12 l^{-a}}{4-\frac32 l^{-a}}>\frac 1 5. 
%           \]

For $r>r^*$ and $c_l(\phi)<c_l$ we have
\[
 1\ge \frac{r^{N_l}\left ( r^{N_l}- c_l(\phi) \right )}
       {r^{2N_l}-2r^{N_l}c_l(\phi)+1}  > \frac{1-\frac12 l^{-a}}{4-\frac32 l^{-a}}>\frac 1 4. 
           \]

Hence
\beqs
%\label{eq:19n}
j_2 \eqsim
  2^{-l}  \int_{r^{*}}^{1-N_l^{-2}}         \frac {dr}{(1-r)}.
\eeqs
Relation \eqref{eq:18an} now  follows from
\[
 2^{-l}  \int_{1-N_l^{-1}}^{1-N_l^{-2}}          \frac {dr}{(1-r)} = \frac{\log 2}{2}; \quad
  2^{-l}    \int_{1-N_l^{-1}}^{r^{*}}         \frac {dr}{(1-r)} \to 0, \ \mbox{as} \ l \to \infty.
 \]

%\bigskip

\noindent{\em Proof of (\ref{eq:12a})} can be done in a similar
way if one chooses
$r^*$ so that
$ (r^*)^{N_{\lfloor t \rfloor}}=1-\frac 12 \lfloor t
\rfloor^{-a}
$
and repeats the above reasonings.

%\bigskip

\noindent{\em Proof of (\ref{eq:13n}) and (\ref{eq:13a})}. These inequalities
can be proved in the same manner so we restrict ourselves to (\ref{eq:13n})
only.

 First we estimate the sum $\sum_{j>l+1} |i_{l,j}(\phi)|$.
Inequality (\ref{eq:03}) implies that for $l>l_0$
\[
\sum_{j>l+1} |i_{l,j}(\phi)| \lesssim 2^{-l}
\]
and the contribution of this sum can be done arbitrary small by choosing $l_0$ large enough.

It remains to estimate
\beqs
% \label{eq:22}
i_{l,l+1}(\phi) =
 \int_{1-N_l^{-1}}^{1-N_l^{-2}} 2^{l+1}
        \frac{r^{N_{l+1}}\left ( r^{N_{l+1}}- c_{l+1}(\phi) \right )}
       {r^{2N_{l+1}}-2r^{N_{l+1}}c_{l+1}(\phi)+1}
       \ \frac {dr}{(1-r) \left (
                \log \frac 1 {1-r}                                                           \right )^2}.
\eeqs
Since $r^{N_{l+1} }< 2^{-1}$ for $r\in (1-N_l^{-1}, 1-N_l^{-2})$ we have
\begin{multline*}
| i_{l,l+1}(\phi) |  \lesssim 2^{-l}  \int_{1-N_l^{-1}}^{1-N_l^{-2}}
      \frac{r^{N_{l+1}}}{1-r}dr= \\
             2^{-l}  \int_{1-N_l^{-1}}^{\rho_l}
      \frac{r^{N_{l+1}} dr}{1-r}  +   2^{-l}  \int_{\rho_l}^{1-N_l^{-2}}
      \frac{r^{N_{l+1}} dr}{1-r},
\end{multline*}
here $\rho_l=1-\log\frac{1}{\ve} N_l^{-2}\in (1-N_l^{-1}, 1-N_l^{-2})$.  Then for all sufficiently large $l$ we obtain
$r^{N_{l+1}}< 2\ve$   when  $r\in (1-N_l^{-1}, \rho_l)$   and thus the first integral is less than $\const\,\ve$.
Finally,
\[
2^{-l}  \int_{\rho_l}^{1-N_l^{-2}}
      \frac{ dr}{1-r}=2^{-l}\log\log\frac 1{\ve}  < \ve,
\]
when $l$ is large enough.
Choosing $\ve$ small enough and $l_0=l_0(\epsilon)$ large enough,    we obtain
the desired inequality  (\ref{eq:13n}). 
%$| i_{l,l+1}(\phi) | < \vk$ for $l>l_0$, $l_0$ being
%sufficiently large.
%Similar estimate hold for $\tilde{i}_j$ when $j=\lfloor t\rfloor +1$.

%\medskip

This completes the proof of Theorem  \ref{th:contr}

\subsection{Concluding remark}

It is interesting to compare  the behavior of the function defined by \eqref{eq:01n} with the results on radial growth obtained in \cite{BLMT}.
Calculations of the last section show that there exists a set $F\subset[0,1]$ such that
\[|F\cap[r,1]|=O\left((1-r)^2\log\left(\frac 1{1-r}\right)^c\right)\quad (r\rightarrow 1)\]
and
\[
\liminf_{r\nearrow 1,\ r\not\in F}\frac{u(re^{i\phi})}{\log\frac 1{1-r}}>0\]
for almost each $\phi\in(-\pi,\pi]$.
Thus function $u$ growth as $\log \frac1{1-r}$ along almost every radius when we delete a system of very thin rings from the unit disc.

\end{document}